\documentclass[12pt]{wart}
\usepackage{xspace,amssymb,amsfonts,euscript,eufrak,amsthm,amsmath}
\usepackage{graphicx,ifpdf}
\ifpdf \usepackage{epstopdf}
 \usepackage{pdfsync}\fi
\usepackage{palatino}

\title[Hochschild homology of Soergel bimodules]
{A geometric model for Hochschild\\  homology of Soergel bimodules}
\author{\textsc{Ben Webster}&\textsc{Geordie Williamson}}
\address{School of Mathematics& Mathematisches Institut\\
         Institute for Advanced Study&der Universit\"at Freiburg\\
         Princeton, NJ 08540&Freiburg 79104\\
         USA& Germany\\}
\email{bwebste@math.berkeley.edu}
\urladdr{http://math.berkeley.edu/\~{}bwebste/}
\email{geordie.williamson@math.uni-freiburg.de}
\urladdr{http://home.mathematik.uni-freiburg.de/geordie/}

\keywords{Soergel bimodules, Hochschild homology, Khovanov-Rozansky homology}

\input xy
\xyoption {all}

  \newcommand{\nc}{\newcommand}
  \newcommand{\renc}{\renewcommand}
\nc{\xox}[2]{\textsf{#2}} 
\nc{\kos}[2]{\EuScript{K}_{#1,#2}}
\nc{\bet}{b}
\nc{\vp}{\varphi}
\nc{\betT}{b_T}
\nc{\de}{\delta}
\nc{\QT}{Q}
\nc{\HT}{S}
\nc{\ep}{\epsilon}
\nc{\Bi}{\mathbf{i}}
\nc{\BB}{B\times B}
\nc{\C}{\mathbb{C}}
\nc{\R}{\mathbb{R}}
\nc{\Cs}[1]{\underline{\C}_{#1}}
\nc{\IH}{I\!H}
\nc{\IC}{\mathbf{IC}}
\nc{\Gw}{G_w}
\nc{\up}{\vp^G_H}
\nc{\Kw}{K_w}
\nc{\B}{\mathcal{B}}
\nc{\SU}[1]{\mathrm{SU}(#1)}
\nc{\SL}[1]{\mathrm{SL}(#1)}
\nc{\GL}[1]{\mathrm{GL}(#1)}
\nc{\HU}[1]{\mathbb{H}\mathrm{U}(#1)}
\nc{\rank}{\mathrm{rk}}
\renc{\t}{\mathfrak t}
\nc{\td}{\t^*}
\nc{\g}{\mathfrak g}
\nc{\HG}{H_G}
\renc{\k}{\mathbf{k}}
\renc{\S}[1]{R_{#1}}
\nc{\Si}{\S i}
\nc{\hc}{\mathbb{H}^*}
\nc{\mc}{\mathcal}
\nc{\Hom}{\mathrm{Hom}}
\nc{\ti}{\tilde}
\renc{\O}{\mathcal {O}}
\nc{\hcBB}{\hc_{B\times B}}
\nc{\si}{\sigma}
\nc{\al}{\alpha}
\nc{\HH}{H\!H}
\nc{\Tor}{\mathrm{Tor}}
\nc{\KR}{\mc{KR}}
\nc{\Supp}{\mathrm{Supp}}
\nc{\ASu}{\mathrm{Supp}'}
\nc{\tri}{\tau}
\nc{\ext}{\mathrm{ext}}
\nc{\baet}[1]{\bar{e}(#1,t)}
\nc{\bht}[1]{\bar{h}(#1,t)}
\nc{\A}{\mathcal{A}}
\nc{\AH}{\A_H}
\nc{\AG}{\A_G}
\nc{\Lotimes}{\stackrel{L}{\otimes}}
\nc{\be}{\beta}
\nc{\ga}{\gamma}
\nc{\BS}[1]{G_{#1}}
\nc{\BSi}{\BS\Bi}
\nc{\Bf}{\mathbf{f}}
\nc{\excise}[1]{}

\nc{\End}{\mathrm{End}}
\nc{\dgmod}{\mathrm{dgMod}}
\nc{\res}{\mathrm{res}}
\nc{\ind}{\mathrm{ind}}
\nc{\For}{\mathrm{For}}

  \newtheorem{defi}{Definition}
  \newtheorem{thm}{Theorem}[section]
  \newtheorem{lem}[thm]{Lemma}
  \newtheorem{deth}[thm]{Definition/Theorem}
  \newtheorem{prop}[thm]{Proposition}
  \newtheorem{cor}[thm]{Corollary}
  
  \theoremstyle{remark}
  \newtheorem{remark}{Remark}
  
\begin{document}
\begin{abstract}
  An important step in the calculation of the triply graded link
  homology theory of Khovanov and Rozansky is the determination of
  the Hochschild homology of Soergel bimodules for $\SL n$.  We
  present a geometric model for this Hochschild homology for any
  simple group $G$, as equivariant intersection homology of $B\times
  B$-orbit closures in $G$.  We show that, in type A these orbit
  closures are equivariantly formal for the conjugation $T$-action. We
  use this fact to show that in the case where the corresponding orbit
  closure is smooth, this Hochschild homology is an exterior algebra
  over a polynomial ring on generators whose degree is explicitly
  determined by the geometry of the orbit closure, and describe its
  Hilbert series, proving a conjecture of Jacob Rasmussen.
\end{abstract}
\maketitle

\section{Introduction}
\label{sec:introduction}

In this paper, we consider the Hochschild homology of Soergel
bimodules, and construct a geometric interpretation of it.  This will
allow us to explicitly compute the Hochschild homology of a special
class of Soergel bimodules.

Soergel bimodules are bimodules over a polynomial ring, which appear
naturally both in the study of perverse sheaves on flag varieties and
of the semiring of projective functors on the BGG category $\O$.
Recently interest in them has been rekindled by the appearance of
connections with link homology as shown by Khovanov \cite{Kho05}.

Khovanov's work showed that one aspect of Soergel bimodules which had
not been carefully studied up to that date was, in fact, of great
importance: their Hochschild homology.  While the operation of taking
Hochschild homology is hard to motivate from a representation
theoretic perspective, we argue that it is, in fact, naturally
geometric, rather than purely combinatorial/algebraic.

Let $G$ be a connected reductive complex algebraic group, with Lie
algebra $\g$.
Let $B$ be a Borel of $G$,
$T$ a Cartan subgroup of $G$, $\mathfrak t$ be its Lie algebra, $n=\dim T$ be the rank of $G$, 
and $W=N_G(T)/T$ be the Weyl group of $G$. For any $w\in W$, we let
$\Gw=\overline{BwB}$.  Note that $T$ is a deformation retract of $B$,
so, we have $H^*_B(X)\cong H^*_T(X)$ for all $B$-spaces $X$.  We will
freely switch between $B$- and $T$-equivariant cohomology throughout
this paper.

We note that $\Gw$ is a closed subvariety of $G$ which is smooth if
and only if the corresponding Schubert variety $\B_w\subset G/B=\B$
is (and thus $\Gw$ typically singular).  Note that $B\times B$ acts on $\Gw$ by left and right
multiplication, and that restricting to the diagonal, we get the action
of $B$ by conjugation.  Of course, we also have left and right actions
of $B$ but we will never consider these separately.  We will {\bf
  always} mean the conjugation action.

Now consider the graded ring $\HT=H_T^*(pt,\C)=\C[\td] $, which is
endowed with the obvious $W$-action.  
Given a simple reflection $s$, denote by
$\S{s}$ the bimodule $\HT \otimes_{\HT^{s}} \HT[1]$ where $S^s$ is the subring of invariants under the reflection $s$ and, by convention, $[a]$ denotes the grading shift $(M[a])^i=M^{i+a}$.

We now come to the definition of Soergel bimodules:

\begin{defi} 
  A {\bf Soergel bimodule} is (up to grading shift) a direct summand
  of the tensor product $\S{\Bi} = \S s \otimes_S \S t \otimes_S \dots
  \otimes_S \S u$ in the category of graded $S$-bimodules, where $\Bi
  = (s, t, \dots, u)$ is a sequence of (not necessarily distinct)
  simple reflections.
\end{defi}

 The motivation for studying Soergel bimodules comes from
 representation theory, geometry and connections between the two. In
 fact, one has the following theorem:

 \begin{thm} {\rm (Soergel)} The indecomposable
   Soergel bimodules are parameterized (up to a shift in the grading)
   by the Weyl group \cite[Satz 6.14]{Soe05}. 

The irreducible Soergel bimodule corresponding
   to $w \in W$ may be obtained as
$\S w\cong I\!H^*_{\BB}(\Gw)$ \cite[Lemma 5]{Soe01}, where $I\!H^*$
  denotes intersection cohomology as defined by Goresky and MacPherson
  \cite{GM1}.
\end{thm}

We choose our grading conventions on $I\!H^*$ so that the 0th degree
is the mirror of Poincar\'e duality, that is, so that the lowest
degree component of $\S w$ sits in degree $-\ell(w)$.

More generally, Soergel bimodules can be identified with the
$\BB$-equivariant hypercohomology $\hcBB(G,\mc F)$ of sums $\mc F$ of
shifts of $\BB$-equivariant semi-simple perverse sheaves on
$G$.

Our central result is a geometric description of Hochschild homology $\HH_*(-)$
of these modules.

\begin{thm}\label{sec:introduction-1}
Let $\mc F$ be a semi-simple $\BB$-equivariant perverse sheaf on $G$ and let $R$ be its $\BB$-equivariant cohomology. 
  Then 
  \begin{equation*}
    \HH_*(\S{})\cong \hc_{B}(G,\mc F),
  \end{equation*}
  where $B$ acts by conjugation (i.e., by the diagonal inclusion
  $B\hookrightarrow\BB$). In particular, we have
  \begin{equation*}
    \HH_*(\S w)\cong\IH_{B}^*(\Gw).
  \end{equation*}
\end{thm}

Unfortunately, $\IH_{B}^*(\Gw)$ has a single grading, whereas
$\HH_*(\S w)$ has two: by decomposition into the components $\HH_i$
(``the Hochschild grading''), and one coming from the grading on $\S
w$ (``the polynomial grading''). This isomorphism takes the single
grading on $\IH_B^*(\Gw)$ to the difference of the two gradings on
$\HH_*(\S w)$.

We can give a geometric interpretation of these gradings, but in a
somewhat roundabout manner.  A result of Rasmussen \cite{Ras06} shows
that

\begin{thm}\label{intro-e-f}
  Assume $G=\SL n$ or $\GL n$.  Then $\IC({G_w})$ is equivariantly formal.  Thus, the map
  \begin{equation*}
    i^*_T: \IH^*_T(G_w)\to \hc_T(T,\IC({G_w}))\cong H^*_T(T)\otimes \IC({G_w})_e
  \end{equation*}
  is injective and an isomorphism after tensoring with the fraction
  field of $S$.
\end{thm}
Furthermore, since $H^*_T(T,\IC({G_w}))\cong S\otimes_\C
H^*(T,\IC({G_w}))$ by the K\"unneth theorem, we can equip this $S$
module with a bigrading, with an ``equivariant'' grading from the
first factor, and a ``topological'' grading from the second.  

\begin{thm}\label{grading-match}
  The intersection cohomology $\IH^*_B(G_w)$ obtains ``topological''
  and ``equivariant'' gradings by transport of structure from the map
  $i^*$ and $\HH_*(\S w)$ obtains the same by the isomorphism of
  Theorem~\ref{sec:introduction-1}.  These are related to the
  ``Hochschild'' and ``polynomial'' gradings by
  \begin{equation*}
    \deg_t(\ga)=\deg_h(\ga) \hspace{15mm} \deg_e(\ga)=\deg_p(\ga)-2\deg_h(\ga)
  \end{equation*}
  where $\deg_*(x)$ denotes the degree of $x$ in the grading whose name begins with the letter $*$.
\end{thm}

The case where $\Gw$ is smooth is of special interest to us, so any
reader who is unhappy with the presence of intersection cohomology and
perverse sheaves can restrict to that case, in which case intersection
cohomology is canonically isomorphic to \v{C}ech cohomology.

\begin{thm}\label{ind-HH}
  If $\Gw$ is smooth (in any type), then the Hochschild homology of
  $\S w$ is of the form
  \begin{equation*}
    \HH_*(\S w)=\land^\bullet(\ga_1,\ldots, \ga_n)\otimes_{\C}\HT
  \end{equation*}
  where $\{\ga_i\}_{i=1,\ldots,m}$ are generators with
  \begin{equation*}
    \deg_h(\ga_i)=1\hspace{15mm}\deg_p(\ga_i)=2k_i
  \end{equation*}
  for positive integers $k_i$ determined by the geometry of $\Gw$.
\end{thm}
These integers can be calculated using the action of $w$ on the root
system, or in $\SL n$ by presenting $\Gw/B$ as an iterated Grassmannian bundle.

While the indecomposable modules $\S w$ are perhaps most natural from
the perspective of geometry or representation theory, definition (1)
(and the study of knot homology, which we discuss briefly in
Section~\ref{sec:knot-homol-soerg}) encourages us to concentrate on
the modules $ \S{\Bi}$.  We call these particular Soergel bimodules
{\bf Bott-Samelson} for reasons which will be clarified in
Section~\ref{sec:geom-soerg-bimod}.

Bott-Samelson modules are naturally identified with the equivariant homology of the ``groupy'' Bott-Samelson space
\begin{equation*}
  \BSi\cong P_{s}\times_{B} P_{t} \times_{B} \cdots\times_{B}P_{u}.
\end{equation*}
and essentially the analogues of all appropriate theorems connecting the $\BB$-orbit closures in $G$ with Soergel bimodules are true here.

\begin{thm}\label{BS-match} If $G=\SL n$ or $\GL n$, then 
  for all $\Bi$, we have
  \begin{equation*}
    \S{\Bi}\cong H^*_{\BB}(\BSi)\hspace{15mm}\HH_*(\S{\Bi})\cong H^*_{B}(\BSi).
  \end{equation*}
  The $T$-conjugation on $\BSi$ is equivariantly formal, and the injection
  \begin{equation*}
    i^*_T:H^*_T(\BSi)\to H_T^*(\BSi^T)
  \end{equation*}
  induces a bigrading on $H^*_T(\BSi)$ matching that on
  $\HH_*(\S{\Bi})$ as in Theorem~\ref{grading-match}.
\end{thm}

The structure of the paper is as follows: In
Section~\ref{sec:knot-homol-soerg}, we discuss the importance of the
Hochschild homology of Soergel bimodules in knot theory.  In
Section~\ref{sec:hochsch-homol-dg}, we prove
Theorem~\ref{sec:introduction-1}, using the formalism of dg-modules,
the relevant points of which we will summarize in
Section~\ref{sec:equi-derived-cat}.  In
Section~\ref{sec:geom-soerg-bimod}, we will cover in more detail how
to construct Soergel bimodules as equivariant intersection cohomology
of various varieties.  Finally, in Section~\ref{sec:equiv-form}, we
prove Theorems \ref{intro-e-f}\,--\,\ref{BS-match}.

\section*{Acknowledgments}
\label{sec:acknowledgments}

The authors would like to thank Catharina Stroppel, Olaf Schn\"urer,  Allen Knutson, Shrawan Kumar, Peter Fiebig,  Michel Brion and Tom Baird
for their helpful suggestions and Wolfgang Soergel for discussions
which greatly enhanced their understanding of the equivariant derived
category.

They would also like to thank the organizers of
the conference ``Microlocal and Geometric Methods in Representation
Theory'' and the staff of Schloss Reisensburg, where the authors first
began their collaboration, as well as Nicolai Reshetikhin, Henning
Haar Andersen and J{\o}rgen Ellegaard Andersen for making possible
their visits to {\AA}rhus, where this work progressed to its current
state.

B.W. was supported under an NSF Graduate Fellowship and NSF Postdoctoral
Fellowship, a Clay Liftoff Fellowship, the RTG grant DMS-0354321,
and a Danish National Research Foundation Niels Bohr Professorship
grant.

G.W. was supported by an Eleanor Sophia Wood traveling scholarship and a Liegrits predoc scholarship. This paper forms part of the his PhD thesis.

\section{Knot homology and Soergel bimodules}
\label{sec:knot-homol-soerg}

While Soergel bimodules merit study simply by being connected so much
delightful mathematics, we have applications in knot theory in mind,
as we will now briefly describe.  The interested reader can find more
details in the papers of Khovanov \cite{Kho05}, 
Rasmussen \cite{Ras06}, and Webster \cite{Web06a}.


The braid group $B_G$ of $G$ is a finitely presented group, with generators $\si_s$ for each simple reflection $s\in W$, which is defined by the presentation
\begin{align*}
 \si_s\si_t&=\si_t\si_s & \big(\text{when }(st)^2=e\big)\\
 \si_s\si_t\si_s&=\si_t\si_s\si_t & \big(\text{when }(st)^3=e\big)\\
(\si_s\si_t)^2&=(\si_t\si_s)^2 & \big(\text{when }(st)^4=e\big)\\
(\si_s\si_t)^3&=(\si_t\si_s)^3 & \big(\text{when }(st)^6=e\big)
\end{align*}
Note that if $G=\SL n$, then $B_n=B_G$ is the standard braid group familiar
from knot theory.

There are several natural weak actions of the braid group on category
$\mc O$ by families of functors (see, for example, \cite{AS,KM}),
which have an avatar on the bimodule side of the picture in the form
of a complex of bimodules attached to each braid group element.  The
description of these bimodule complexes can be found in various
sources, for example \cite{Kho05}, or for general Coxeter groups in
\cite{Rou04}.

Define the complexes of $S$-bimodules:
\begin{align*}
  F(\si_s)&=\cdots\longrightarrow \HT[-1]\longrightarrow \S{s}\longrightarrow  \hspace{2.5mm}0 \hspace{2.5mm}\longrightarrow \cdots\\
  F(\si_s^{-1})&=\cdots\longrightarrow \hspace{4mm}0 \hspace{4mm}\longrightarrow \S{s}\longrightarrow \HT[1]\longrightarrow \cdots
\end{align*}
where the maps between non-zero spaces are the unique (up to scalar)
non-zero maps of degree 0.  These maps are defined by (respectively) the pushforward and pullback in $\BB$-equivariant cohomology for the inclusion $B\hookrightarrow \overline{BsB}$.

\begin{thm} The shuffling complex
  \begin{equation*}
    F(\si)=F(\si_{i_1}^{\ep_1})\otimes_{\HT}\cdots \otimes_{\HT}F(\si_{i_m}^{\ep_m})
  \end{equation*}
  of a braid $\si=\si_{i_1}^{\ep_1}\cdots\si_{i_m}^{\ep_m}$ (where $\ep_i=\pm 1$) depends up to homotopy only on $\si$ not its factorization.  In particular,
  \begin{equation*}
    F(\si\si')=F(\si)\otimes_{\HT}F(\si'),
  \end{equation*}
  so $F$ defines a categorification of $B_G$.
\end{thm}

The maps in this complex also have a geometric interpretation: each
degree is a direct sum of Bott-Samelson modules for subsequences of
$\Bi$, and the ``matrix coefficients'' of the differential between
these are induced by pullback or pushforward maps on $\BB$-equivariant
cohomology for inclusions $\BS{\Bi'}\to\BS{\Bi''}$ of Bott-Samelson
spaces where $\Bi'$ and $\Bi''$ are subsequences of $\Bi$ which differ
by a single index.

Even better, this complex can be used to find a knot invariant, as was
shown by Khovanov \cite{Kho05}.  Let $\HH_*(\S{})$ be the Hochschild
homology of $\S{}$, which can be defined (using the standard equivalence
between $\HT-\HT$-bimodules with $\HT\otimes\HT^{op}$) by
\begin{equation*}
  \HH_i(\S{})=\Tor^i_{\HT\otimes\HT^{op}}(\HT,\S{}).
\end{equation*}
This can be calculated by the Hochschild complex of $\HT$ (which is
often used as a definition), or by the Koszul complex, both of which
are free resolutions of $\HT$ as an $\HT\otimes\HT^{op}$-module.

In the case where $G=\SL n$, Hochschild homology is a categorification
of the trace on the braid group defined by Jones \cite{Jon87}.
Remarkably, combining these creates a categorification of knot
polynomials, which had previously been defined by Khovanov and
Rozansky.

\begin{thm}
  {\rm (Khovanov, \cite{Kho05})} As a graded vector space, the homology $\KR(\bar\si)$ of the
  complex $\HH_i(F( \si))$ is depends only on the knot $\bar\si$, and
  in fact, is precisely the triply graded homology defined by Khovanov
  and Rozansky in \cite{KR05}.
\end{thm}

Applying Theorem~\ref{sec:introduction-1} to our remarks above, we can understand the differentials of the complex $\HH_i(F( \si))$ in terms of pullback and pushforward on $B$-equivariant cohomology.

\section{The equivariant derived category and dg-modules}
\label{sec:equi-derived-cat}

Since our readers may be less well-acquainted with the formalism of
equivariant derived categories and their connection with dg-modules,
as developed by Bernstein and Lunts, in this section, we will provide
a brief overview of the necessary background for later sections.  This
material is discussed in considerably more detail in their monograph
\cite{BL}.


Suppose a Lie group $G$ operates on a space $X$. We have maps:
\begin{align*}
m &\colon G \times X \to X & m(g,x) &= g\cdot x \\
\pi &\colon G \times X \to X & \pi(g,x) &=x
\end{align*}
A function $f$ on $X$ is $G$-invariant if and only if $m^*f =
\pi^*f$. It is therefore natural
to define a $G$-equivariant sheaf on $X$ to be a sheaf $\mathcal{F}$
on $X$ together with an isomorphism $\theta : m^*\mathcal{F} \to
\pi^*\mathcal{F}$. (There is also a cocycle condition that we ignore here).

One can show that if $G$ operates topologically
freely on $X$ with quotient $X / G$ then the categories of
$G$-equivariant sheaves on $X$ and sheaves on $X/G$ are equivalent.

Faced with a $G$-space, one would wish to define an ``equivariant derived category''. This should
associate to a pair $(G,X)$ a triangulated category $D^b_G(X)$ together
with a ``forgetting $G$-equivariance'' functor $D^b_G(X) \to D^b(X)$. For
any reasonable definition of $D^b_G(X)$, there should be an equivalence
$D^b_G(X) \cong D^b(X/G)$ if $G$ acts topologically freely as well as 
notions of pullback and pushforward for equivariant maps.



The trick is to notice that, at least homotopically, we may assume
that the action is free: we ``liberate'' $X$ (i.e. make it free) by
replacing it with $X \times EG$ where $EG$ is the total space of the
universal $G$-bundle (i.e. any contractible space on which $G$ acts
freely). The first projection $p : X \times EG \to X$ is a homotopy
equivalence (because $EG$ is contractible) and the diagonal operation
of $G$ on $X \times EG$ is free. Thus, we can consider the quotient
map $q : X \times EG \to X \times_G EG$ as ``liberation'' of $X\to
X/G$. The following definition then makes sense:

\begin{defi} 
The \emph{(bounded) equivariant derived category} $D^b_G(X)$ is the
  full subcategory of $D^b(X \times_G EG)$ consisting of complexes
  $\mathcal{F} \in D^b(X \times_G EG)$ such that $q^*\mathcal{F} \cong
  p^*\mathcal{G}$ for some complex $\mathcal{G} \in D^b(X)$. 
\end{defi}

In the case where $X$ is a single point, we have $X\times_G EG\cong
EG/G$, which is usually denoted $BG$.

\begin{remark} 
  This is not exactly Bernstein and Lunts' definition. Consider the
  following diagram of spaces:
  \begin{equation*}
    X \stackrel{p}{\gets} X \times EG \stackrel{q}{\to} X \times_G EG
  \end{equation*}
  They define an equivariant sheaf to be a tuple $(\mc G, \mc F,\al)$
  where $\mathcal{G}\in D^b(X)$, $\mathcal{F}\in D^b(X\times_G EG)$,
  and $\alpha: p^*\mathcal{G}\to q^*\mathcal{F}$ is an isomorphism.
  However, the functor to the above definition
  which forgets everything except for $\mathcal{G}$ is an equivalence
  of categories.
\end{remark}

\begin{remark} As previously mentioned, it is natural to expect a
  ``forgetting $G$-equivariance functor'' $\For : D^b_G(X) \to
  D^b(X)$. With $p$ and $q$ as in the previous remark, we may define
  $\For(\mathcal{F}) = p_*q^*\mathcal{F}$. If $1 \hookrightarrow G$ is
  the inclusion of the trivial group, then $\For(\mathcal{F}) \cong
  \res_G^1 \mathcal{F}$ ($\res_G^1$ is defined below). \end{remark}

  If $H\subset G$ is a subgroup, then we may take $EG$ for $EH$ and we
  have a natural map
\begin{equation*}
\up\colon X\times_H EG\to X\times_G EG
\end{equation*}
commuting with the projection
to $X$.
The pullback and pushforward by this map induce functors
\begin{equation*}
  (\up)^*=\res^G_H\colon D^b_G(X)\to D^b_H(X)\hspace{6mm}(\up)_*=\ind_G^H\colon D^b_H(X)\to D^b_G(X)
\end{equation*}

Similarly, for any $G$-space we have a map $X \times_G EG \to BG$. If $X$ is a reasonable space (for example a complex algebraic variety with the classical topology) push-forward yields a functor
\begin{displaymath}
\pi_* : D^b_G(X) \to D^b_G(pt)
\end{displaymath}
which commutes with the induction and restriction functors.

In this work we are interested in equivariant cohomology for connected Lie
groups. These emerge as the cohomology of objects living in
$D_G(pt)$. The first key observation of Bernstein and Lunts' is
the following:

\begin{prop} If $G$ is a connected Lie group then $D^b_G(pt)$ is the triangulated subcategory of $D^b(BG)$ generated by the constant sheaf. \end{prop}

It turns out that this observation allows Bernstein and Lunts to give an algebraic description of $D_G(pt)$. For this we need the language of differential graded algebras and modules.

\begin{defi} 
  A \emph{differential graded algebra} (or \emph{dg-algebra}) is a
  unital, graded associative algebra $\mc{A} = \oplus_{i \in \mathbb{Z}} A_i$
  together with an additive endomorphism $d : \mc{A} \to \mc{A}$ of degree 1 such
  that:
\begin{enumerate}
\item $d$ is a differential: i.e. $d^2 = 0$.
\item $d$ satisfies the Leibniz rule: $d(ab) = (da)b + (-1)^{\deg a} a(db)$.
\item $d(1_\mc{A}) = 0$, where $1_\mc{A}$ denotes the identity of $\mc{A}$.
\end{enumerate}
A \emph{left differential graded module} (or \emph{left dg-module}) over a differential graded algebra $\mc{A}$ is a graded left $\mc{A}$-module $M$ together with a differential $d_M : M \to M$ of degree 1 satisfying:
\begin{enumerate}
\item $d_M^2 = 0$.
\item $d_M(am) = (da)m + (-1)^{\deg a} a(d_Mm)$.
\end{enumerate}
A morphism of dg-modules is a graded $\mc{A}$-module homomorphism $f: M \to
M^\prime$ commuting with the differentials.
\end{defi}

\begin{remark} 
  If $\mc{A} = A_0$ is concentrated in degree zero, then a
  differential graded module is just a chain complex of
  $\mc{A}$-modules. 
\end{remark}
Given any dg-module $\mc{M}$ we may consider $H^*(\mc{M})$, which is a
graded module over the graded algebra $H^*(\mc{A})$.  As with the
category of modules over an algebra, the category of dg-modules over a
dg-algebra has a homotopy category and a derived category as defined by Bernstein and Lunts \cite{BL}.

\begin{defi}
  A map of dg-modules $f:\mc{M}\to \mc{M}'$ is a {\bf quasi-isomorphism} if the induced map $H^*(\mc{M})\to H^*(\mc{M}')$ on cohomology is an isomorphism.

  The {\bf derived category} of dg-modules for the dg-algebra $\mc{A}$, which we
  denote by $\dgmod{\mc{A}}$, is the category whose objects are dg-modules,
  and whose morphisms are compositions of chain maps and formal
  inverses to quasi-isomorphisms. We denote by $\dgmod^f{\mc{A}}$ the full
  subcategory consisting of dg-modules, finitely generated over $\mc{A}$.
\end{defi}

Given a morphism $\mc{A} \to \mc{A}^{\prime}$ of dg-algebras we would like to
define functors of restriction and extension of scalars. Restriction
of scalars is unproblematic (acyclic complexes are mapped to acyclic
complexes) however a little more care is needed in defining extension
of scalars.  Just as in the normal derived category, one needs a
special class of objects in order to define functors. In $\dgmod{\mc{A}}$
these are the $\mathcal{K}$-projective objects \cite{BL} which we will
not discuss in complete generality. In the sequel we will only be
interested in a special class of dg-algebras in which it is possible
to construct $\mathcal{K}$-projective objects rather explicitly:

\begin{prop}[Proposition 11.1.1 of \cite{BL}] 
  Let $\mc A = \C[x_1, \dots x_n]$ be viewed as a dg-algebra by setting
  $d_{\mc A} = 0$ and requiring that each $x_i$ have even degree. Then all
  dg-modules which are free as $\A$-modules are
  $\mathcal{K}$-projective.
\end{prop}

Let us now describe how to construct a $\mathcal{K}$-projective
resolution of a dg-module $\mc{M}$ when $\mc A$ is as in the proposition.
Assume first that $\mc{M}$ has zero differential. We may choose a free
resolution in the category of graded $\mc A$-modules:
\begin{displaymath}
P_{-n} \to \dots \to P_{-2} \to P_{-1} \to \mc{M}
\end{displaymath}
We then consider the direct sum $\mc{P} = \oplus_i P_i[i]$ as a dg-module with the
natural differential. For example, elements in $P_{-2}$ are mapped
under $d_P$ into $P_{-1}$ using the corresponding differential in the
above resolution. The natural morphism $\mc{P} \to
\mc{M}$ (again induced from the resolution above) is a quasi-isomorphism.

Using standard techniques from homological algebra, one may give a
more elaborate construction of a $\mc K$-projective resolution for
dg-modules over $\mc{A}$ with non-trivial differential \cite{BL}, but this will not be
necessary for our results.


We may now define the extension of scalars functor. Suppose we have a
morphism $\mc{A} \to \mc{A}^{\prime}$ of dg-algebras, and that $\mc{A}$ and
$\mc{A}^{\prime}$ are as in the proposition. For any $\mc{M} \in \dgmod
\mc{A}^{\prime}$ and $N \in \dgmod \mc{A}$ we define
\begin{displaymath}
\ext_\mc{A}^{\mc{A}^{\prime}}(N) = \mc{A}^{\prime} \stackrel{L}{\otimes}_\mc{A} N =  \mc{A}^{\prime} \otimes_{\mc{A}} P
\end{displaymath}
where $P$ is a $\mathcal{K}$-projective resolution of $P$. (Note that we may also take a $\mathcal{K}$-projective resolution of $\mc{A}^{\prime}$ as an $\mc{A}$ dg-module).

We can now return to a discussion of the equivariant derived
category. Abelian and triangulated categories can often by described by
``module categories'' over endomorphism rings of generators. We have
already seen that $D^b_G(pt)$ is precisely the triangulated subcategory
of $D^b(BG)$ generated by the constant sheaf. Bernstein and Lunts then
consider the functor $\Hom_{D(BG)} ( \Cs{BG}, -)$ and notice
that $\End(\Cs{BG})$ has the structure of a
dg-algebra. Moreover there is a quasi-isomorphism $\A_G=H^*(BG)
\to \End(\Cs{BG})$ of dg-algebras.

This yields a functor
\begin{equation*}
\Gamma_G=\Hom_{D^b(BG)} ( \Cs{BG}, -)\colon D^b_G(pt) \to \dgmod^f \A_G.
\end{equation*}
Bernstein and Lunts then show:

\begin{thm}[Main Theorem of Bernstein-Lunts \cite{BL}]\label{equiv-thm}  Assume as above that $G$ is a connected Lie group. 
The above functor gives an equivalence commuting with the cohomology functor:
\begin{equation*}
\Gamma_G : D_G^b(pt) \to \dgmod^f \A_G
\end{equation*}
Moreover if $\varphi : G \to H$ is an inclusion 
of connected Lie groups and $\A_H \to \A_G$ is the induced homomorphism the
 restriction and induction functors have an algebraic description in terms of dg-modules:
\begin{equation*}
\xymatrix{
D^b_G(pt) \ar[r]^{\ind_G^H} \ar[d]^{\Gamma_H} & D^b_H(pt) \ar[d]^{\Gamma_G} & &
D^b_H(pt) \ar[r]^{\res_{H}^G} \ar[d]^{\Gamma_H} & D^b_G(pt) \ar[d]^{\Gamma_G} \\
\dgmod^f {A}_G \ar[r]^{\text{res. of sc.}} & \dgmod^f \mathcal{A}_H & & 
\dgmod^f \mathcal{A}_H \ar[r]^{\ext^{\AG}_{\AH}}& \dgmod^f {A}_G
}
\end{equation*}
\end{thm}

\begin{remark} 
  Note that, if $G$ is a connected Lie group $\A_G$ is
  always a polynomial ring on even generators, and hence, by the
  previous discussion, we can always construct a sufficient supply of
  finitely generated $\mathcal{K}$-projective objects. 
\end{remark}

We will now describe equivariant intersection cohomology complexes,
which will be important in the sequel. Given a variety $X$ (for
simplicity assumed to be over the complex numbers), a smooth locally
closed subvariety $U$, and a local system $\mathcal{L}$ on $U$ there is a complex $\IC(U) \in D^b(X)$ called the intersection
cohomology complex extending $\mathcal{L}$, with remarkable properties
(see for example \cite{BBD} and \cite{GM1}).

It is possible to construct equivariant analogues of the intersection
cohomology complexes, as described in Chapter 5 of \cite{BL}: If $X$
is furthermore a $G$-variety for a complex algebraic group $G$, $U$ is
a smooth $G$-stable subvariety, and $\mathcal{L}$ is a $G$-equivariant local system on $U$ then there exists an ``equivariant intersection
cohomology complex'' which we will also denote $\IC(U, \mathcal{L})$. Forming intersection cohomology complexes behaves well
with respect to restriction, as the following lemma shows:

\begin{lem} \label{res-lemma} 
  If $H \hookrightarrow G$ is an inclusion
  of algebraic groups, $X$ is a $G$-variety, $U$ is a smooth $G$-stable subvariety then:
  \begin{displaymath}
    \res_G^H \IC(U, \mathcal{L}) \cong \IC(U, \res_G^H \mathcal{L})
  \end{displaymath} 
\end{lem}

\begin{remark} In dealing with equivariant intersection cohomology complexes it is more convenient to work with an equivalent definition of the equivariant derived category as a limit of categories associated to $X \times_G E_n G$, where $E_nG$ is an finite dimensional algebraic variety, which approximates $EG$ (\cite{BL}). \end{remark}

\section{Hochschild homology and dg-algebras}
\label{sec:hochsch-homol-dg}

Recall that the Hochschild homology of an $\HT\otimes S$-module $R$ can be defined as
\begin{equation*}
  \HH_*(\S{})=\HT \Lotimes_{\HT\otimes \HT} \S{}
\end{equation*}
where $\HT$ has been made into an $\HT\otimes\HT$ algebra by the
multiplication map.  Since $\HT\cong\A_{B}$ and
$\HT\otimes\HT\cong\A_{\BB}$, this map is that induced by the diagonal
group homomorphism $B\hookrightarrow\BB$.  Thus, we expect that the
geometric analogue of taking Hochschild homology is restricting from a
$\BB$-action to the diagonal $B$.

However, we must be careful about the difference between dg-modules
and modules.  Hochschild homology is an operation on
$\HT-\HT$-bimodules, not dg-bimodules.  Thus, to make a precise
statement requires us to restrict to formal complexes.


\begin{defi} 
  Let $M \in \dgmod {\mc A_G}$ be a dg-module. If $M \cong H^*(M)$
  in $\dgmod{\mathcal{A}_G}$ we say that $M$ is {\bf formal}. Similarly,
  $\mathcal{F} \in D_G(pt)$ is {\bf formal} if $\Gamma_G(\mathcal{F})$
  is. 
\end{defi}

The following proposition connects the Hochschild cohomology of formal
equivariant sheaves with another equivariant cohomology. This is our
main technical tool.

\begin{prop} \label{formal-isom}
  Suppose $\mathcal{F} \in D_{B \times B}(pt)$ is formal. Then one has
  an isomorphism:
\begin{equation*}
\bigoplus_{i} \HH_i( \hc_{B \times B}(\mathcal{F}))[i] \cong \hc_B( \res_{B \times B}^B \mathcal{F})
\end{equation*}
as graded $S$-modules. 

Furthermore this isomorphism is functorial.  That is, if $\mc F$ and
$\mc G$ are formal sheaves in $ D_{B \times B}(pt)$, and $\vp:\mc
F\to\mc G$ is a morphism, the maps $\hc_{B}(\vp)$ and $\HH_*(\hc_{\BB}(\vp))$ commute with this isomorphism.
\end{prop}

\begin{proof}[Proof of Proposition~\ref{formal-isom}] In order to work out the Hochschild homology of $\hc_{B \times B}(\mathcal{F})$ we may take a free resolution of $\hc_{B \times B}(\mathcal{F})$ by $S \otimes S$-modules:
\begin{equation*}
0 \to P_{-2n} \to \dots \to P_{-1} \to \hc_{B \times B}(\mathcal{F})
\end{equation*}
We then apply $S \otimes_{S \otimes S} -$ and take
cohomology. However, because $ \hc_{B \times B}(\mathcal{F}) \cong
\Gamma_{B \times B}(\mathcal{F})$ in $\dgmod\mathcal{A}_{B \times B}$, we
may also regard $\bigoplus P_i[-i]$ as a $\mathcal{K}$-projective
resolution of $\Gamma(\mathcal{F})$. By Theorem \ref{equiv-thm} we have:
\begin{equation*}
\hc_B(\res_{B \times B}^B \mathcal{F}) \cong H^*( S \overset{L}{\otimes}_{S \otimes S}  \hc_{B \times B}(\mathcal{F})) \cong \bigoplus \HH_i( \hc_{B \times B}(\mathcal{F}))[i]
\end{equation*}
\end{proof}

\begin{proof}[Proof of Theorem~\ref{sec:introduction-1}] 
  Let $\mathcal{F}$ denote the image of the intersection cohomology
  sheaf on $\overline{BwB}$ in $D_{B \times B}(pt)$. We will see in the next section
  that $\mathcal{F}$ is formal. Hence we can apply the above
  proposition. However, we also know that $\hc_{B \times
    B}(\mathcal{F})$ is the indecomposable Soergel bimodule $\S w$.
  Hence:
  \begin{equation*}
    \bigoplus_i \HH_i(\S w)[i] \cong \hc_B (\res_{B \times B}^B \mathcal{F}) 
  \end{equation*}

But $\res_{B \times B}^B$ commutes with the map to a point and $\res_{B \times B}^B(\IC(G_w)) \cong \IC(G_w)$ (Lemma \ref{res-lemma}). Hence:
\begin{displaymath}
\bigoplus_i \HH_i(\S w)[i] \cong \hc_B(\IC(G_w)) 
\end{displaymath}
  This then yields the main theorem.
\end{proof}

\section{The geometry of Bott-Samelson bimodules}
\label{sec:geom-soerg-bimod}

In this section, we discuss Bott-Samelson bimodules, calculate their
$B \times B$-equivariant cohomology and obtain the formality results
needed above.

Since we have already described the geometric realization of
indecomposable bimodules, as intersection cohomology of subvarieties
of $G$, we know abstractly that the Bott-Samelson bimodule $\S\Bi$
must be the hypercohomology of a perverse sheaf obtained by taking a
direct sum of $\IC{}$-sheaves of these subvarieties with appropriate
multiplicities.

However, this is deeply dissatisfying from a geometric viewpoint, and
totally at odds with our viewpoint that Bott-Samelson bimodules are
very natural objects. Thus we would like a more natural geometric
realization of them.

For each simple reflection $s$, let $P_s$ be the minimal parabolic containing $s$. For a sequence $\Bi = (s, t,  \dots, u)$ of simple reflections, let
\begin{equation*}
\BSi=P_s\times_{B} P_t \dots \cdots\times_{B}P_u.
\end{equation*}
We call this the {\bf Bott-Samelson variety} corresponding to $\Bi$.
Note that this variety still carries a $B\times B$-action, and thus a
diagonal $B$-action.

\excise{
The Bott-Samelson case is noticeably more complex than that of a smooth orbit closure, as is indicated by the appearance of combinatorics in describing the fixed point set of $\BSi$.



 \begin{prop}\label{BS-fixed-points}
   The subvariety $G_\Bi^T$ is the disjoint union of the left $T$-orbits
   \begin{equation*}
     F_{\boldsymbol{\ep}}=T_{\C}\cdot (s_{i_1}^{\ep_1},\ldots,s_{i_m}^{\ep_m})
   \end{equation*}
   for all sequences $\boldsymbol{\ep}=(\ep_1,\ldots,\ep_m)$ with $\ep_i\in\{0,1\}$ and $s_{i_1}^{\ep_1}\cdots s_{i_m}^{\ep_m}=e$.
   If we let $\tri(\Bi)$ be the number of such
   subsequences, then $\bet(G_\Bi^T)=2^n\tri(\Bi)$
 \end{prop}
 \begin{proof}
   We note that $G_\Bi\cong \mc B_{\Bi}\times_{\mc B} G$, so
   $G_\Bi^T\cong \mc B_{\Bi}^T\times_{\mc B} G^T$.  Thus, $G_\Bi^T$ is
   the disjoint union of one copy of the torus for each fixed point of
   $\mc B_\Bi$ whose image is the standard Borel.  These are in
   bijection with subsequences of $\Bi$ whose product in the Weyl
   group is the identity.  
 \end{proof} }

Furthermore, we have a projective $B\times
B$-equivariant map $m_\Bi:\BS{\Bi}\to G$ given by multiplication, intertwining the diagonal $B$-action on $G_\Bi$ with the conjugation $B$-action on $G$.  

The quotient of $\BS{\Bi}$ by the right Borel action is the familiar
projective Bott-Samelson variety which is used to construct
resolutions of singularities for Schubert varieties.  It is worth
noting that just like in the flag variety case, if $\Bi$ is a reduced
expression (i.e. if $\ell(st\dots u)$ is the length of $\Bi$), then the multiplication map is a resolution of singularities.

Let us explain how to calculate the $B \times B$-equivariant
cohomology of the Bott-Samelson varieties. Actually, we will calculate
the corresponding dg-module over $S \otimes S$. We start with a lemma:

\begin{lem} Suppose $\mathcal{F} \in D_{\BB}(pt)$ and let $s$ be a simple reflection. Then:
\begin{displaymath}
  \Gamma_{\BB}( \res_{P_s \times B}^{\BB} \ind_{\BB}^{P_s \times B} \mathcal{F}) = \HT\otimes_{\HT^s} \Gamma_{\BB}(\mathcal{F})
\end{displaymath}
\end{lem}

\begin{proof} Thanks to Theorem \ref{equiv-thm} we know that
  $\Gamma_{P_s \times B} (\ind_{\BB}^{P_s \times B} \mathcal{F})$ is
  equal to $\Gamma_{\BB}(\mathcal{F})$ regarded as an $\HT^s \otimes
  \HT$-module. Hence 
  \begin{equation*}
    \Gamma_{\BB}( \res_{P_s \times B}^{\BB}
  \ind_{\BB}^{P_s \times B} \mathcal{F}) = (\HT \otimes \HT)
  \Lotimes_{\HT^s \otimes \HT} \Gamma_{\BB}(\mathcal{F}).
    \end{equation*}
    However, $S \otimes S$ is free as a module over $S^s\otimes S$ and
    is hence $\mathcal{K}$-projective. Thus the derived tensor product
    coincides with the naive tensor product and the result follows.
\end{proof}

We can now prove the crucial ``formality'' claim mentioned above:

\begin{prop} 
  The direct images of the sheaves $\Cs{\BS{\Bi}}$ and $\IC(\Gw)$ in $D_{B \times
    B}(pt)$ are formal.
\end{prop}

\begin{proof}
First notice that we
can can write the sheaf  $(m_{\Bi})_*\Cs{\BS{\Bi}}$ as a iterated induction and
restriction:
\begin{equation*}
(m_{\Bi})_*\Cs{\BS{\Bi}}\cong \res_{P_{s} \times B}^{B \times B}\ind_{B \times B}^{P_{s} \times B} \dots 
\res_{P_{t} \times B}^{B \times B}\ind_{B \times B}^{P_{t} \times B}(\Cs{B})
\end{equation*}
Hence by the above lemma, letting $p$ be the projection to a point:
\begin{equation*}
\Gamma_{B \times B} (p_*\Cs{\BS{\Bi}}) = S\otimes_{S^{s}} S \otimes_{S^{t}} \otimes \dots \otimes_{S^{u}} S \qquad \text{in $D_{\A_{\BB}}$}
\end{equation*}
Thus the proposition is true for $\Cs{\BS{\Bi}}$. Now, by the decomposition theorem of \cite{BBD} or more precisely, its equivariant version in \cite{BL},
  we may obtain $\IC(\Gw)$ as a direct summand of $(m_{\Bi})_*\BS{\Bi}$, if $\Bi  = (s, \dots, s)$ is a reduced expression for $w$. Thus $p_*\IC(\Gw)$ is a direct summand of $p_*\Cs{\BS{\Bi}}$ and is also formal.
\end{proof}

\section{Equivariant formality}
\label{sec:equiv-form}

Now, we will carry out some actual computations of $B$-equivariant
cohomology, and thus of Hochschild homology.


Of course, the best setting in which to compute equivariant cohomology
of a variety is when that variety (or more precisely, the sheaf one
intends to compute the hypercohomology of) is equivariantly formal.

\begin{deth}\label{eq-for-def}
  We call $\mc F\in D^b_T(X)$ on a $T$-variety $X$ equivariantly formal if one of the following equivalent conditions holds:
  \begin{enumerate}
  \item The $S$ module $\hc_T(X,\mc F)$ is free.
  \item The differentials in the spectral sequence
    \begin{equation*}
      \hc(X,\mc F)\otimes\HT\Rightarrow\hc_{T}(X,\mc F)
    \end{equation*}
    are trivial, that is, if $\hc(X,\mc F)\otimes\HT\cong\hc_{T}(X,\mc
    F)$ as $\HT$-modules.
  \item We have the equality $\dim_\C H^*(X)= \dim_{\C}H^*(X^T)$.
  \end{enumerate}
\end{deth}

The equivariant formality of the Bott-Samelsons of $\SL n$ has been
proven by Rasmussen in different language.  

\begin{prop}
  If $G=\SL n$ or $\GL n$, the $T$-space $\BSi$ is equivariantly formal for all $\Bi$.
\end{prop}
\begin{proof}
  By Theorem \ref{sec:introduction-1}, $\HH_*(\S{\Bi})$ is free as an
  $S$-module if and only if $H^*_T(\BS{\Bi})$ is.  By \cite[Propositon
  4.6]{Ras06}, the module $\HH_*(\S{\Bi})$ is free in type $A$, so by
  Definition/Theorem~\ref{eq-for-def} above $\BSi$ is equivariantly formal.
\end{proof}
This in turn implies that $m_*\Cs{\BSi}$ is equivariantly formal,
where $m:\BSi\to G$ is the multiplication map.  Since all summands of
equivariant formal sheaves are themselves equivariantly formal, and
each $\IC(\Gw)$ appears as a summand of such a sheaf (if, for example,
$\Bi$ is a reduced word for $w$), this completes the proof of
Theorem~\ref{intro-e-f} and the first part of Theorem~\ref{BS-match}.

While the most obvious consequence of equivariant formality,
calculating the equivariant cohomology from ordinary or vice versa, is
a rather useful one, it has less obvious ones as well.

\begin{prop}{\em (Goresky, Kottwitz, MacPherson \cite[Theorem 6.3]{GKM})}
  If $\mc F$ is equivariantly formal, then the pullback map
  \begin{equation*}
    i^*_T:H_T^*(X,\mc F)\to H_T^*(X^T,\mc F)
  \end{equation*}
  is injective.
\end{prop}

As we mentioned earlier, we are interested in the Hochschild homology
of Soergel bimodules as a bigraded object (so that we get a
triply-graded knot homology theory), but the grading on equivariant
hypercohomology is only one of these.  From now on, we consider
$H^*_T(\Gw)$ as a bigraded $\HT$-module, with the bigrading defined by
the isomorphism with Hochschild homology given by Proposition
\ref{formal-isom}.

\begin{proof}[Proof of Theorems \ref{grading-match} and \ref{BS-match}]
  Since the pullback map $H_T^*(\BSi)\to H_T^*(\BSi^T)$ is induced by
  a map of Soergel bimodules, it is homogeneous in both gradings.
  Similarly with the map induced by the inclusion of a summand
  $\IC(\Gw)\hookrightarrow m_*\Cs{\BSi}$.  Thus we need only establish
  the theorem for $\BSi^T$.  As this is a union of complex tori with
  the trivial action, we need
  only establish the theorem for $T$.

  This case follows immediately from applying $\HH_0$ to the Koszul
  resolution of $H_{T\times T}(T)\cong S$ as a bimodule over itself.
\end{proof}

Let us turn to the case where $\Gw$ is smooth. Since $H^*_T(\Gw)\cong S\otimes H^*(\Gw)$, it will prove very interesting to understand $H^*(\Gw)$.  Surprisingly, no description of this cohomology seems to be in the literature, but in fact there is a very beautiful one.

As is well known (and we reprove in the course of
Lemma~\ref{schub-lemma} below), there exists a unique decreasing
sequence of positive integers $k_1,\cdots k_{n}$ such that the Hilbert series
of $H^*(\Gw/B)$ is of the form
\begin{equation*}
  \sum_{i=1}^{\ell(w)}q^{i/2}\dim H^i(\Gw/B)=\prod_{j=1}^n\frac{1-q^{k_j}}{1-q}
\end{equation*}

\begin{thm}\label{equiv-form}
  If $\Gw$ is smooth, 
  then as an algebra,
  \begin{equation}\label{eq:1}
    H^*(\Gw)\cong \land^{\bullet}(\ga_1,\ldots \ga_n)
  \end{equation}
  where $\deg(\ga_i)=2k_i-1$, and as an
  $\HT$-algebra,
  \begin{equation*}
    H^*_T(\Gw)\cong\HH_i(\S w)\cong \HT\otimes H^*(\Gw).
  \end{equation*}
  In the standard double grading on $\HH_i(\S w)$, we have $\deg(1\otimes \ga_i)=(1,2k_i)$.
\end{thm}

This explicitly describes the Hilbert series of $\HH_i(\S w)$, proving
a conjecture of Rasmussen.
\begin{cor}
  The Hilbert series of $\HH_i(\S w)$ is given by
  \begin{equation*}
    \sum_{i,j}a^iq^j \HH_i(\S w)_{2j}=\prod_{\ell=1}^{n}\frac{1+aq^{k_\ell}}{1-q}
  \end{equation*}
\end{cor}
Since $\dim H^*(\Gw)=2^{\rank G}$, Definition/Theorem~\ref{eq-for-def}(3) establishes the equivariant formality of $\Gw$ independently of the earlier results of this paper (and for all types).  

As usual in Lie theory, we define the height $h(\al)=\langle\rho^\vee, \al\rangle$ of a root $\al$ to be its evaluation against the fundamental coweight.

\begin{lem}\label{schub-lemma}
  The cohomology ring $H^*(\Gw/B)$ is a quotient of the polynomial
  ring $\HT$ by a regular sequence $(f_1,\ldots, f_{n})$.  Define
  $k_i=\deg(f_i)$.

  The number of times the integer $m$ appears in the list $k_1,\ldots,
  k_{n}$ is precisely the number of roots in $R^+\cap w^{-1}(R^-)$
  of height $m-1$ minus the number of such roots of height $m$, where
  $R^+$ is the set of positive roots of $G$ and $R^-=-R^+$.
\end{lem}
\begin{proof}
  For ease, assume $\Gw$ is not contained in a parabolic subgroup.  By
  results of Aky\i ld\i z and Carrell \cite{AC89}, the ring $H^*(\Gw/B)$
  is a quotient of $\C[BwB/B]$ by a regular sequence of length
  $\ell(w)$ by a regular sequence $(g_1,\ldots, g_{k})$ where $k=\ell(w)$.

  In this grading, $\C[BwB/B]$ is a polynomial algebra generated by
  elements of degree $2h(\al)$ for each root $\al\in R^+\cap
  w^{-1}(R_-)$, and the degrees of $g_i$ are given by $2h(\al)+2$ as
  $\al$ ranges over the same roots.  Corresponding to the simple roots
  are $n-1$ generators $x_1,\ldots x_{n}$ of degree 2, which are the
  first Chern classes of line bundles on $G/B$ corresponding the
  fundamental weights, and for the other roots we have $\ell(w)-n+1$
  other generators $y_n,\cdots y_{\ell(w)}$ of higher degree.  Here,
  we assume these are in increasing order by degree.  

  It is a well known fact that the cohomology ring $H^*(G/B)$ is
  generated by the Chern classes $x_i$.  Since the natural pullback
  map $H^*(G/B)\to H^*(\Gw/B)$ is onto, the ring $H^*(\Gw/B)$ is as
  well.  That is, if $p=\deg y_{k}>2$, then
  \begin{equation*}
    y_k-\sum_{\deg(g_i)=p}\be_ig_i\in \sum_{\deg(g_j)<p} \HT g_j.
  \end{equation*}
  Since $y_k\notin \sum_{\deg(g_j)<p}\HT g_j$, we can eliminate $y_k$
  and any single relation $g_i$ such that $\be_i\neq 0$.  Obviously, 
  $(g_1,\ldots,g_n)\setminus\{g_i\}$ is again a regular sequence in
  $\C[x_1,\ldots,x_{n},y_n,\ldots, y_{k-1}].$ Applying this argument
  inductively, we obtain a subsequence $(g_{i_1},\ldots, g_{i_{n}})$
  which is regular in $\HT$, which is the desired sequence.

  Since the number of relations of degree $j$ which have been
  eliminated is the number of generators of degree $j$, the remaining
  number of relations is precisely the difference between these, which
  is also the number of roots of height $j/2-1$ minus the number of
  height $j/2$ in $R^+\cap w^{-1}(R^-)$.
\end{proof}
\begin{proof}[Proof of Theorem~\ref{equiv-form}]
  Applying the Hirsch lemma (as stated in the paper \cite{DGMS}) to the
  fibration $B\to \Gw\to \Gw/B$, we see that the cohomology ring
  $H^*(\Gw)$ is the cohomology of the dg-algebra
  $H^*(\Gw/B)\otimes_\HT\EuScript{K}_T$ where $\EuScript{K}_T$ is the
  Koszul complex of $\HT$ (the natural free resolution of $\C$ as an
  $\HT$-module equipped with the Yoneda product).

  Since $(f_1,\ldots, f_{n})$ is regular, $H^*(\Gw/B)$ is
  quasi-isomorphic to the Koszul complex $\EuScript{K}_{\Bf}$.  Thus,
  the cochain complex of $\Gw$ is quasi-isomorphic to
  $\EuScript{K}_{\Bf}\otimes_\HT\C$.  This is just an exterior algebra
  over $\C$ with generators $\gamma_1,\ldots, \gamma_{n}$ with
  degree given by $\deg(\gamma_i)=2k_i-1$.

  Since we used a Koszul complex, the generators $\ga$ land in
  Hochschild degree 1 under the restriction map to $T$, so the degree
  of $\gamma_i$ is $(2k_i,1)$.

  If $w\in W'$, where $W'$ is a proper parabolic subgroup of $W$ (with
  corresponding Levi subgroup $G'\subset G$), and $W'$ is the minimal
  such parabolic.  Then $G_w\cong G'_w\times T/T'\times N\cap
  w_0'Nw_0'$.  Since $N\cap w_0'Nw_0'$ is unipotent and thus
  contractible, by K\"unneth, we have
  \begin{equation*}
\dim_\C H^*(G_w)\cong \dim_\C H^*(G_{w'})\dim_\C H^*(T/T')=2^{\dim T'}2^{\dim T/T'}=2^{\dim T}. \qedhere
\end{equation*}
Thus, $G_w$ is equivariantly formal and the degrees $k_i$ for $G_w$
are simply those of $G_{w'}$ extended by adding 1's.
\end{proof}


\begin{thebibliography}[DGMS]{}
\providecommand\bibmarginpar{\leavevmode\marginpar}
\def\urlstyle#1{{\tt #1}}

\bibitem[AC]{AC89}
\textbf{E Aky{\i}ld{\i}z}, \textbf{J\,B Carrell}, \emph{A generalization of the
  {K}ostant-{M}acdonald identity}, Proc. Nat. Acad. Sci. U.S.A.  \textbf{86} (1989)
  3934--3937

\bibitem[AS]{AS}
\textbf{H Andersen}, \textbf{C Stroppel}, \emph{Twisting functors on $\mathcal{O}$}.  Represent. Theory   \textbf{7}  (2003), 681--699

\bibitem[BBD]{BBD}
\textbf{A\,A Beilinson}, \textbf{J Bernstein}, \textbf{P Deligne},
  \emph{Faisceaux pervers}, from: ``Analysis and topology on singular spaces, I
  (Luminy, 1981)'', Ast\'erisque 100, Soc. Math. France, Paris (1982)  5--171

\bibitem[BL]{BL}
\textbf{J Bernstein}, \textbf{V Lunts}, \emph{Equivariant sheaves and
  functors}, volume 1578 of \emph{Lecture Notes in Mathematics},
  Springer-Verlag, Berlin (1994)

\bibitem[DGMS]{DGMS}
\textbf{P Deligne}, \textbf{P Griffiths}, \textbf{J Morgan}, \textbf{D
  Sullivan}, \emph{Real homotopy theory of {K}\"ahler manifolds}, Invent. Math.
   \textbf{29} (1975) 245--274

\bibitem[GM]{GM1}
\textbf{M Goresky}, \textbf{R MacPherson}, \emph{Intersection homology theory},
  Topology  \textbf{19} (1980) 135--162

\bibitem[Jo]{Jon87}
\textbf{V\,F\,R Jones}, \emph{Hecke algebra representations of braid groups and
  link polynomials}, Ann. of Math. (2)  \textbf{126} (1987) 335--388

\bibitem[KM]{KM}
\textbf{O Khomenko}, \textbf{V Mazorchuk},
\emph{On Arkhipov's and Enright's functors}.
Math. Z.  \textbf{249} (2005), 357--386

\bibitem[Kh]{Kho05}
\textbf{M Khovanov}, \emph{{Triply-graded link homology and Hochschild homology
  of Soergel bimodules}} \xox{arXiv}{arXiv:math.GT/0510265}

\bibitem[KR]{KR05}
\textbf{M Khovanov}, \textbf{L Rozansky}, \emph{{Matrix factorizations and link
  homology II}} \xox{arXiv}{arXiv:math.QA/0505056}

\bibitem[Ra]{Ras06}
\textbf{J Rasmussen}, \emph{{Some differentials on Khovanov-Rozansky homology}}
  \xox{arXiv}{arXiv:math.GT/0607544}

\bibitem[Ro]{Rou04}
\textbf{R Rouquier}, \emph{{Categorification of the braid groups}} (2004)
  \xox{arXiv}{arXiv:math.RT/0409593}

\bibitem[So1]{Soe90}
\textbf{W Soergel}, \emph{Kategorie {$\mathcal{ O}$}, perverse {G}arben und
  {M}oduln \"uber den {K}oinvarianten zur {W}eylgruppe}, J. Amer. Math. Soc.  \textbf{3}
  (1990) 421--445

\bibitem[So2]{Soe92}
\textbf{W Soergel}, \emph{The combinatorics of {H}arish-{C}handra bimodules},
  J. Reine Angew. Math.  \textbf{429} (1992) 49--74

\bibitem[So3]{Soe01}
\textbf{W Soergel}, \emph{Langlands' philosophy and Koszul duality}, Algebra---representation theory (Constanta, 2000),  379--414,
NATO Sci. Ser. II Math. Phys. Chem., 28, Kluwer Acad. Publ., Dordrecht, 2001.

\bibitem[So4]{Soe05}
\textbf{W Soergel}, \emph{Kazhdan-Lusztig-Polynome und unzerlegbare Bimoduln \"uber Polynomringen}, Journal of the Inst. of Math. Jussieu  \textbf{6}(3) (2007) 501--525


\bibitem[We]{Web06a}
\textbf{B Webster}, \emph{{Khovanov-{R}ozansky homology via a canopolis formalism}}, Algebr. Geom. Topol.  \textbf{7}  (2007), 673--699.

\bibitem[GKM]{GKM}
\textbf{M Goresky}, \textbf{R Kottwitz}, \textbf{R MacPherson},
\emph{Equivariant cohomology, {K}oszul duality, and the localization
  theorem}, Invent. Math.  \textbf{131} (1998), 25--83.
\end{thebibliography}

\def\cprime{$'$}

\end{document}